\documentclass[12pt]{article}
\usepackage{amssymb,latexsym,amsmath}
\usepackage{amsthm}
\usepackage{geometry}
\usepackage{graphicx,color}
\usepackage{lineno}
\newtheorem{theorem}{Theorem}

\newtheorem{lemma}[theorem]{Lemma}

\theoremstyle{plain}

\theoremstyle{definition}
\newtheorem{remark}[theorem]{Remark}
\newtheorem{example}[theorem]{Example}
\newtheorem{definition}{Definition}[section]

\numberwithin {equation}{section}

\usepackage{float}
\restylefloat{table}

\begin{document}

\title{Interception Conditions in Multi-Agent Pursuit Games Governed by Linear Dynamics with Gronwall-type Constraints}
\author{David Terna. Gbande $^{a}$\thanks{princedavison4@gmail.com}\and
Abbas Ja'afaru Badakaya$^{b}$\thanks{ajbadakaya.mth@buk.edu.ng} \and
Mehdi Salimi$^{c}$\thanks{mehdi.salimi@kpu.ca (Corresponding author)}\and
Aminu Sulaiman Halliru $^{d}$\thanks{aminuhalliru1@gmail.com}\and
Jamilu Adamu $^{e}$
\thanks{jamiluadamu88@gmail.com } 
 } \maketitle
\begin{center}
$^{a}$ Department of Statistics, Federal Polytechnic Wannune, Benue State, Nigeria\\
$^{b,d}$ Department of Mathematical Sciences, Bayero University, Kano, Nigeria \\
$^{c}$ Mathematics Department, Kwantlen Polytechnic University, BC, Canada. \\
$^{e}$ Mathematics Department, Federal University, Gashua, yobe State, Nigeria.
\end{center}
\maketitle
 
\begin{abstract}
\noindent This paper investigates a multi-agent pursuit problem in which a group of pursuers seeks to intercept one or more evaders within a fixed operational time horizon. The motion of both pursuers and evaders is described by first-order linear differential equations, and the control inputs available to the agents are subject to Gronwall-type constraints. Pursuit is considered successful when at least one pursuer reaches the position of each evader at a finite time within the prescribed duration. Within this framework, we construct permissible strategies for the pursuers and derive sufficient conditions under which interception is guaranteed despite the evaders' counteractions. To support the theoretical results, we present numerical examples that demonstrate the effectiveness of the proposed strategies and confirm the analytical conditions for successful pursuit.
 \end{abstract}
\maketitle
\section{Introduction}
\noindent Differential games provide a rigorous mathematical framework for analyzing strategic interactions among multiple decision makers whose actions jointly determine the evolution of a dynamical system over continuous time. The dynamics are described by differential equations, while each player selects admissible control strategies to optimize an individual objective that may conflict with those of other players. As a natural extension of optimal control theory to multi-agent systems, differential games have become an indispensable tool for studying competitive and cooperative decision-making in engineering, economics, robotics, and military operations. Since the pioneering works of Isaacs \cite{REFM6} and Krasovskii \cite{REFM7}, the theory has developed into a mature research area with pursuit-evasion differential games serving as one of its central applications. These foundational studies introduced fundamental concepts such as admissible strategies, value functions, and optimality principles, which continue to underpin modern developments in differential game theory.

\noindent Pursuit-evasion games constitute one of the most extensively studied classes of differential games. In these games, one or more pursuers attempt to capture, intercept, or continuously track one or more evaders, whose objective is to avoid capture for as long as possible \cite{REF30a}. Owing to their strong mathematical foundations and broad practical relevance, pursuit-evasion models have found applications in robotics, cybersecurity, autonomous systems, biological processes, and artificial intelligence; see, for example, \cite{TY2A,TY69A}. Motivated by increasingly realistic applications, recent research has focused on multi-player formulations involving several pursuers and one or more evaders. For instance, Salimi and Ferrara \cite{REFM9} investigated pursuit strategies for one evader and multiple pursuers, deriving structural properties of optimal controls and sufficient conditions for capture. Ahmed \cite{REFM1} studied pursuit games involving finitely many players evolving in a closed convex subset of $\mathbb{R}^{n}$ under generalized integral constraints, while Badakaya et al. \cite{REFM2} analyzed multi-pursuer games in the Hilbert space $l_{2}$, establishing sufficient conditions for $l$-capture and constructing piecewise pursuit strategies. Infinite-dimensional formulations have also attracted considerable attention. Umar et al. \cite{REFM12} investigated pursuit problems governed by infinite systems of binary differential equations in $l_{2}$ under geometric constraints, deriving guaranteed capture-time estimates, whereas Ibragimov et al. \cite{REFM5} considered pursuit games with multiple pursuers and one evader under special classes of integral constraints.

\noindent An important direction in the literature concerns the admissible constraints imposed on the players' control functions. One of the earliest contributions in this direction is due to Berkovitz \cite{REFM3}, who generalized classical pursuit problems by introducing integral constraints on the controls. This formulation significantly broadened the class of admissible strategies and provided a realistic representation of systems with limited energy or fuel resources. Subsequently, Borodo et al. \cite{REFM2B} studied a pursuit differential game in the sequence space $l_{\infty}$ with finitely many pursuers and one evader, where all controls satisfy component-wise integral constraints, and established sufficient conditions guaranteeing capture. Besides integral constraints, geometric constraints have been widely adopted to model bounded instantaneous control capabilities. More recently, Gronwall-type constraints have emerged as a flexible and powerful framework for modeling systems in which the admissible controls depend on accumulated effects or feedback dependent bounds. These constraints have been investigated in both finite and infinite dimensional settings, revealing analytical properties that differ substantially from those associated with classical geometric and integral constraints. In particular, Rilwan and co-authors \cite{REFM8} established the existence of admissible strategies, sufficient conditions for guaranteed pursuit, and characterizations of game values under Gronwall-type constraints, while emphasizing the additional analytical challenges involved in constructing coordinated pursuit strategies. More recently, Gbande et al. \cite{REFM4} considered pursuit games involving many players under combined geometric and Gronwall-type constraints and obtained sufficient conditions for successful pursuit.

\noindent Another important line of research concerns pursuit games whose dynamics are governed by first-order linear differential equations. Such dynamics provide a more realistic representation of systems involving damping, decay, or linear feedback effects than the simple motion model commonly used in the literature. In this direction, Usman et al. \cite{REFM11} investigated a pursuit differential game with first-order linear dynamics in which the players' controls satisfy both geometric and integral constraints. They established sufficient conditions guaranteeing successful pursuit in each considered scenario, thereby extending the classical theory to a broader class of dynamical systems.

\noindent Although substantial progress has been made in the study of pursuit differential games, the combined treatment of multi-player systems governed by first-order linear differential dynamics under Gronwall-type control constraints remains largely unexplored. Existing studies have typically focused either on first-order linear dynamics with geometric and integral constraints or on Gronwall-type constraints for different classes of dynamics. Consequently, explicit sufficient conditions for guaranteed pursuit in multi-player first-order linear systems subject to Gronwall-type constraints are still lacking.

\noindent The present paper addresses this gap by investigating a pursuit differential game involving multiple pursuers and multiple evaders whose motions are governed by first-order linear differential equations under Gronwall-type control constraints. We develop admissible pursuit strategies and establish explicit sufficient conditions guaranteeing successful pursuit. In this way, the paper extends previous results on pursuit-evasion differential games, unifies aspects of the existing literature on first-order linear dynamics and Gronwall-type constraints, and contributes new theoretical results to the analysis of multi-player differential games.

\section{Problem Statement}
\noindent This Research work extends the theory of constrained differential games beyond the classical framework of simple motion player dynamics, motivated by the pursuit-evasion frameworks of Rilwan et al. \cite{REFM8} and Ahmed et al. \cite{REFM1} which studied fixed-duration and optimal-time pursuit games under Gronwall-type control constraints using simple linear dynamics, this research addresses the problem of analyzing pursuit differential games involving multiple pursuers and many evaders governed by modified first-order linear differential equations, with the aim of establishing conditions for capture under Gronwall-type constrained controls.
\noindent We consider a differential game problem in  $\mathbb{R}^n$ with the pursuers $p_{i}$ and 
evaders $e_{j}$ move in accordance with the following differential equations:
\begin{equation}\label{AT31}
\left\{ \begin{array}{ll}
\dot{p}_{i}(t) +\beta_{i} p_{i}(t)  = c_{i}(t), \ \ \ p_{i}(0) = p_{i0},\ \ \ i= 1,2,3,...,l,  \\
  \dot{e}_{j}(t) +\beta_{j} e_{j}(t) = f_{j}(t), \ \ \ e_{j}(0) = e_{j0},\ \ \ j= 1,2,3,...,m,
\end{array}\right
.
\end{equation}
where $  c_{i}(\cdot) , f_{j}(\cdot) $ are the control functions of the pursuers and evaders respectively,
 $\beta_{i}$ and $\beta_{j},$ are  non-zero constants and that $ l$ and $ m$ are natural numbers such that $ l > m $ .
 \begin{definition}
 A measurable function $c_{i}(t) = (c_{i1}(t) ,c_{i2}(t), \dots,  c_{in}(t))  $  that satisfy the inequality
 \begin{equation}\label{DAV19}
||c_{i}(t)||^{2} \leq \varrho_{i}^{2} + 2k \int_{0}^{t}||c_{i}(s)||^{2}ds  \ \ \ t \geq 0\ \ \ i= 1,2,3,...,l
\end{equation}
where $\varrho_{i}$ and $ k $ are positive numbers, is called permissible control of the  pursuers $p_{i}$ with respect to Gronwall-type constraint.
\end{definition}
\begin{definition}
A measurable function $f_{j}(t) = (f_{j1}(t) ,f_{j2}(t), \dots,  f_{jn}(t))  $  that satisfy the inequality
 \begin{equation}\label{DAV22}
 ||f_{j}(t)||^{2} \leq \varsigma_{j}^{2} + 2k \int_{0}^{t}||f_{j}(s)||^{2}ds  \ \ \ t \geq 0 \ \ \ j= 1,2,3,...,m
\end{equation}
where $\varsigma_{j} $ and $ k $ are positive numbers, is called permissible control of the evader $ e_{j} $ with respect to Gronwall-type constraint.
\end{definition}

\begin{definition}
 A function $C_{i}(\cdot)$ is called strategy of the  pursuers, if for any permissible control  of the evaders the system
  \begin{equation}\label{AMIMM}
\left\{ \begin{array}{ll}
 \dot{p}_{i}(t) + \beta_{i} p_{i}(t)= C_{i}(t), \ \ \ p_{i}(0) = p_{i0}, i= 1,2,3,...,l\\
  \dot{e}_{j}(t)+\beta_{j} e_{j}(t) = \ f_{j}(t), \ \ \ e_{j}(0) = e_{j0}, j= 1,2,3,...m,
\end{array}\right.
\end{equation}
has a unique solution ($p_{i}(t), e_{j}(t) $).
\end{definition}
\begin{definition}
	Pursuit is said to be completed  at some finite time $0 < \tau \leq \vartheta $ in the game (\ref{AT31})- (\ref{DAV22}) if there exist strategies $c_{i}(t)$ of the pursuers $p_{i}$ such that for any permissible control $f_{j}(\cdot)$ of the evader $e_{j}$ the equation $p_{i}(\tau) = e_{j}(\tau)$ is satisfied for some $ i \in \left\lbrace 1,2,\cdots l \right\rbrace $ and all $ j \in \left\lbrace 1,2,3,\cdots m \right\rbrace $
\end{definition}

Given the players permissible control $ c_{i}(\cdot)$ and $ f_{j}(\cdot),$ the corresponding 
paths $ p_{i}(t) ,e_{j}(t),$ at any time $ t > 0 $ of the players for any initial positions $ p_{i0}, e_{j0} $, respectively are given by
\begin{align}\label{C21}
	p_{i}(\vartheta) = p_{i0}e^{-\beta_{i}\vartheta} + e^{-\beta_{i}\vartheta} \int_{0}^{\vartheta}c_{i}(s)e^{\beta_{i}s} ds ,
\end{align}
\begin{align}\label{C4}
	e_{j}(\vartheta) = e_{j0}e^{-\beta_{j} \vartheta} + e^{-\beta_{j}\vartheta} \int_{0}^{t}f_{j}(s)e^{\beta_{j}s} ds.
\end{align}
\begin{lemma}\label{TB}\cite{REFM8}
Let   $ w(t),  t \geq 0$ be a measurable function,$ \varphi $ and $ k $ be non-negative real numbers  then the inequality
\begin{equation}
 \displaystyle|w(t)|\leq \varphi e^{kt} ,
\end{equation}
is true whenever
\begin{equation}
|w(t)|^2\leq \varphi^2 + 2k\int_{0}^{t}|w(s)|^2ds .
\end{equation}
\end{lemma}
\textbf{Research question}:\\
 What is (are) the sufficient condition(s) that can ensure completion of pursuit for a finite time in the game (\ref{AT31})- (\ref{DAV22})?
\section{Main Result}
In this section, we present the main results of the research work in form of theorems and their proofs.

Let the $i^{th} $ pursuer chases $j^{th} $ evader and uses the strategy
 \begin{equation}\label{PSS}
c_{i}(t) =\dfrac{(e^{\Lambda_{ij} \vartheta}e_{j0} -p_{i0})}{\vartheta }e^{-\beta_{i} t} + e^{\Lambda_{ij}( \vartheta - t)}f_{j}(t),
 \end{equation}
where $ \Lambda_{ij} = \beta_{i}-\beta_{j} ,\ \ \ 0\leq t\leq \vartheta  .$\\
Before presenting the main results and theorems, we introduce several preliminary definitions that will be used: $\varsigma: = \max \{ \varsigma_{j}\}$. $\varrho: = \min \{ \varrho_{i}\} . $ \ \ \ $\gamma: = \displaystyle\max_{i,j} \{\frac{\|e^{\Lambda_{i,j}}e_{j0} -p_{i0}\|}{\vartheta}\}.$\ \ \ $ \Lambda: = \displaystyle\max_{i,j}\{\Lambda_{i,j} \}. $ \ \ \ $ \beta = \displaystyle\max_{i}\{\beta \} ,$ \ \ \ $\phi = \displaystyle\max \{ -\beta , k \} .$ \\
 The subsequent theorems establish rigorous answers to the research question formulated in this work.
 \begin{theorem}\label{D1} 
Consider the game defined by (\ref{AT31})- (\ref{DAV22}) under the assumptions that 
 $\Lambda =  0$ and $ \varrho > \gamma $.
If the $i^{th}$ pursuer adopts the strategy given in (\ref{PSS}), then pursuit is guaranteed whenever \\
$$\varsigma \leq\left\{ \begin{array}{ll}
 e^{-k\vartheta}\left(\varrho -\gamma\right) , \ \ \  \beta \geq 0.   \\
 \varrho e^{-\phi\vartheta} -\gamma ,    \ \ \    \beta < 0 .\end{array}\right.
$$
\end{theorem}
\begin{proof}
When $\Lambda = 0$, we have $\beta_{i}  = \beta_{j}$, and consequently the strategy(\ref{PSS}) reduces to
 \begin{equation}\label{PS}
c_{i}(t) =\dfrac{(e_{j0} -p_{i0})}{\vartheta }e^{-\beta_{i} t} + f_{j}(t).
 \end{equation}
If the $i^{th} $ pursuer apply strategy (\ref{PS}) using (\ref{C21}) and (\ref{C4}), gives  $ p_{i}(\vartheta)= e_{j}(\vartheta).$ Indeed,
\begin{align*}
p_{i}(\vartheta)&= p_{i0}e^{-\beta_{i}\vartheta}+ \int_{0}^{\vartheta} e^{\beta_{i}(s-\theta)}c_{i}(s)ds \nonumber \\
&= p_{i0}e^{-\beta_{i}\vartheta}+ \int_{0}^{\vartheta}e^{\beta_{i}(s-\vartheta)}\left(\dfrac{(e_{j0} -p_{i0})}{\vartheta }e^{-\beta_{i} s} + f_{j}(s)\right)ds \nonumber \\
&=p_{i0}e^{-\beta_{i}\vartheta}+\int_{0}^{\vartheta}\dfrac{(e_{j0}-p_{i0})}{\vartheta}e^{-\beta_{j}\vartheta}ds + \int_{0}^{\vartheta} e^{\beta_{i}(s-\vartheta)}f_{j}(s)ds \nonumber \\
&=p_{i0}e^{-\beta_{i}\vartheta}+ \dfrac{(e_{j0}-p_{i0})}{\vartheta}e^{-\beta_{j}\vartheta}\int_{0}^{\vartheta}ds + \int_{0}^{\vartheta} e^{\beta_{i}(s-\vartheta)}f_{j}(s)ds \nonumber \\
&=e_{j0}e^{-\beta_{j}\vartheta}+ \int_{0}^{\vartheta} e^{\beta_{j}(s-\vartheta)}f_{j}(s)ds
=e_{j}(\vartheta).
\end{align*} 
\noindent The permissibility of the pursuer strategy (\ref{PS}) can be established using Lemma \ref{TB}, as follows:\vspace{-.2cm}
\begin{eqnarray}\label{kk1}
\|c_{i}(t)\|^{2}&=&\left\|\dfrac{(e_{j0} -p_{i0})}{\theta }e^{-\beta_{i} t} + f_{j}(t)\right\|^{2}\nonumber \\
&=&\dfrac{e^{-2\beta_{i} t}\|e_{j0} -p_{i0}\|^2}{\vartheta^2}+ 2 \left\langle \dfrac{e^{-\beta_{i} t}(e_{j0} -p_{i0})}{\vartheta } , f_{j}(t)\right\rangle  +\|f_{j}(t)\|^{2}\nonumber \\
&\leq &\dfrac{e^{-2\beta_{i} t}\|e_{j0} -p_{i0}\|^2}{\vartheta^2}+\frac{2e^{-\beta_{i} t}\|e_{j0} -p_{i0}\|\|f_{j}(t)\|}{\vartheta}+\|f_{j}(t)\|^{2}\nonumber \\
&\leq &\dfrac{e^{-2\beta t}\|e_{j0} -p_{i0}\|^2}{\vartheta^2}+\frac{2e^{-\beta t}\|e_{j0} -p_{i0}\|\|f_{j}(t)\|}{\theta} \nonumber \\
&+ & \|f_{j}(t)\|^{2}.
\end{eqnarray}
If $\beta \geq 0$, then from (\ref{kk1}) we have \vspace{-.2cm}
\begin{align*}
\|c_{i}(t)\|^{2} &\leq  \dfrac{\|e_{j0} -p_{i0}\|^2}{\vartheta^2}+\frac{2\|e_{j0} -p_{i0}\|\varsigma_{j} e^{kt}}{\vartheta}+\varsigma_{j}^{2} e^{2kt} \\
&\leq \dfrac{\|e_{j0} -p_{i0}\|^2}{\vartheta^2}+\frac{2\|e_{j0} -p_{i0}\|\varsigma_{j} e^{k\vartheta}}{\vartheta}+\varsigma_{j}^{2} e^{2k\vartheta} \\
&\leq  \dfrac{\|e_{j0} -p_{i0}\|^2}{\vartheta^2}+\frac{2\|e_{j0} -p_{i0}\|\varsigma e^{k\vartheta}}{\vartheta}+\varsigma^{2} e^{2k\vartheta} \\
&\leq \dfrac{\|e_{j0} -p_{i0}\|^2}{\vartheta^2}+\frac{2\|e_{j0} -p_{i0}\| e^{k\vartheta}}{\vartheta} \left(\varrho -\gamma\right)e^{-k\vartheta} +e^{-2k\vartheta}\left(\varrho -\gamma\right)^{2} e^{2k\vartheta} \\
&= \dfrac{\|e_{j0} -p_{i0}\|^2}{\vartheta^2}+\frac{2\|e_{j0} -p_{i0}\| }{\vartheta} \left(\varrho -\gamma\right) +\left(\varrho -\gamma\right)^{2} \nonumber \\
&=\dfrac{\|e_{j0} -p_{i0}\|^2}{\vartheta^ 2}+\frac{2\varrho\|e_{j0} -p_{i0}\| }{\vartheta}-\frac{2\|e_{j0} -p_{i0}\| \gamma}{\vartheta}+\varrho^2 -2\varrho \gamma +\gamma^2 \nonumber \\
&\leq \dfrac{\|e_{j0} -p_{i0}\|^2}{\vartheta^2}+\frac{2\varrho\|e_{j0} -p_{i0}\|}{\vartheta}-\frac{2\|e_{j0} -p_{i0}\|^2}{\vartheta^2}+ \varrho^2 -\frac{2\varrho\|e_{j0} -p_{i0}\|}{\vartheta}+\frac{\|e_{j0} -x_{i0}\|^2}{\theta^2}\nonumber \\
& = \varrho^2 \leq \varrho_{i}^2 + 2k\int_{0}^{\vartheta}\|c_{i}(t)\|^{2}dt .
\end{align*}
If $\beta < 0$, then from (\ref{kk1}) we have
\begin{align*}
\|c_{i}(t)\|^{2}& = \dfrac{e^{2\phi t}\|e_{j0} -p_{i0}\|^2}{\vartheta^2}+\frac{2e^{\phi t}\|e_{j0} -p_{i0}\|\varsigma e^{\phi t}}{\vartheta}+\varsigma^{2} e^{2\phi t}\\
&\leq\dfrac{e^{2\phi \vartheta}\|e_{j0} -p_{i0}\|^2}{\vartheta^2}+\frac{2e^{2\phi \vartheta}\|e_{j0} -p_{i0}\|\varsigma}{\vartheta}+\varsigma^{2} e^{2\phi\vartheta}\\
&\leq\dfrac{e^{2\phi\vartheta}\|e_{j0} -p_{i0}\|^2}{\vartheta^2}+\frac{2e^{2\phi\vartheta}\|e_{j0} -p_{i0}\|}{\vartheta}\left(\varrho e^{-\phi\vartheta} -\gamma\right)+\left(\varrho e^{-\phi \vartheta} -\gamma\right)^{2} e^{2\phi\vartheta}\\
&=\dfrac{e^{2\phi\vartheta}\|e_{j0} -p_{i0}\|^2}{\vartheta^2}+\frac{2e^{\phi\vartheta}\|e_{j0} -p_{i0}\|\varrho}{\vartheta}-\frac{2e^{2\phi\vartheta}\|e_{j0} -p_{i0}\|\gamma}{\vartheta}
\\& +\varrho^2 -2\gamma\varrho e^{\phi \vartheta} +\gamma^2e^{2\phi \vartheta}\\
&\leq\dfrac{e^{2\phi\vartheta}\|e_{j0} -p_{i0}\|^2}{\vartheta^2}+\frac{2e^{\phi\vartheta}\|e_{j0} -p_{i0}\|\varrho}{\vartheta}-\frac{2e^{2\phi\vartheta}\|e_{j0} -p_{i0}\|^2}{\vartheta^2}
\\& +\varrho^2 -\frac{2e^{\phi\vartheta}\|e_{j0} -p_{i0}\|\varrho}{\vartheta}+\dfrac{e^{2\phi\vartheta}\|e_{j0} -p_{i0}\|^2}{\vartheta^2} \\
& = \varrho^2 \leq \varrho^2 + 2k \int_{0}^{\vartheta}\|c_{i}(t)\|^{2}dt .
\end{align*}
Therefore, the proof of Theorem \ref{D1} is complete.
\end{proof}
To validate the proof of Theorem \ref{D1}, a numerical example is presented in the remark below.
\begin{remark} It is possible to simultaneously satisfy the inequalities of the hypothesis of the Theorem \ref{D1}. View the following example below:
\end{remark}
\begin{example}
 For $  \beta \geq 0 $, we consider $l =2 $ and $m = 1$. Then Theorem \ref{D1} inequalities can be fulfilled by selecting
$ \beta_{1} = 3 ,\beta_{2}  = -4 , \varrho_{1}  = 100, \varrho_{2}  = 200, \varsigma_{1}  = \frac{1}{50} $ ,
$ \vartheta = 2 , k = 3 $
and letting
$ e_{10} = ( e^{1}_{10} , e^{2}_{10}, ..., e^{n}_{10})  = (1,0,...,0 ) ,$
$ p_{10} = ( p^{1}_{10} , p^{2}_{10}, ..., p^{n}_{10})  = (2,0,...,0 ) $ \\
$ p_{20} = ( p^{1}_{20} , p^{2}_{20}, ..., p^{n}_{20}) = (3,0,...,0 ) .$ \\
This implies that
$  \beta   = 3 $ and
$ \frac{\left\|(e_{10} -p_{10})\right\|}{\vartheta} = \frac{(-1^2 +0^2 + ... +0^2 )^{\dfrac{1}{2}}}{2}  = \frac{1}{2}$. Thus
$\gamma = \displaystyle\max_{i,j} \{\frac{\|e^{\Lambda_{i,j}}e_{j0} -p_{i0}\|}{\vartheta}\} = \frac{1}{2}.$
Therefore, we can see  that,
$$ \frac{1}{50} \leq e^{-6}\left(100  -0.5\right). $$
That is $$ \frac{1}{50} < 40141.29,$$
which means that
 $$ \varsigma \leq  e^{-k\vartheta}\left(\varrho -\gamma\right)$$
 With assumed values above, the inequality for $ \beta < 0$ can also be satisfied. This means that
$  \beta   = -4. $ and $  \phi =  \displaystyle\max \{ -4 ,3 \} = 3. $ \\
Therefore, we can see that, $$\frac{1}{50} \leq 100 e^{-6} -0.5. $$
That is  $$ \frac{1}{50} < 0.034$$
implies that,
 $$ \varsigma \leq \varrho e^{-\phi\vartheta} -\gamma .$$
\end{example}

\begin{theorem}\label{D3}
Suppose in the game (\ref{AT31})- (\ref{DAV22}) that
 $\Lambda <  0$ and $\gamma < \varrho $.
If the $i^{th}$ pursuer uses the strategy (\ref{PSS})then pursuit can be completed whenever \\
  $$\varsigma \leq\left\{ \begin{array}{ll}
 e^{-k\vartheta}\left(\varrho -\gamma\right) ,\ \ \ \beta \geq 0.  \\
\varrho e^{-\phi\vartheta} -\gamma,\ \ \ \beta <0.   \end{array}\right
.
$$
\end{theorem}
\begin{proof}
If the $i^{th}$ pursuer apply  strategy (\ref{PSS}) then using (\ref{C21}) and (\ref{C4}), we can deduce that  $p_{i}(\vartheta) = e_{j}(\vartheta)$. Indeed,
\begin{align*}
p_{i}(\vartheta)&= p_{i0}e^{-\beta_{i}\vartheta}+ \int_{0}^{\vartheta} e^{\beta_{i}(s-\vartheta)}c_{i}(s)ds \nonumber \\
&= p_{i0}e^{-\beta_{i}\vartheta}+ \int_{0}^{\vartheta}e^{\beta_{i}(s-\vartheta)}\left(\dfrac{(e^{\Lambda_{ij} \vartheta}e_{j0} -p_{i0})}{\vartheta }e^{-\beta_{i} s} + e^{\Lambda_{ij} (\vartheta-s)}f_{j}(s)\right)ds \nonumber \\
&=p_{i0}e^{-\beta_{i}\vartheta}+\int_{0}^{\vartheta}\dfrac{e_{j0}e^{\beta_{i(s- \theta)}}e^{\Lambda_{ij} \vartheta}e^{-\beta_{i} s}}{\vartheta}ds-\int_{0}^{\vartheta}\dfrac{p_{i0}e^{\beta_{i}(s-\vartheta)}e^{-\beta_{i} s}}{\vartheta}ds
\\& + \int_{0}^{\vartheta}e^{\beta_{i}(s-\vartheta)} e^{\Lambda_{ij} (\vartheta-s)}f_{j}(s)ds \nonumber \\
&=p_{i0}e^{-\beta_{i}\vartheta}+\int_{0}^{\vartheta}\dfrac{e_{j0}e^{-\beta_{j}\vartheta}}{\vartheta}ds-\int_{0}^{\vartheta}\dfrac{p_{i0}e^{-\beta_{i}\vartheta}}{\vartheta}+ \int_{0}^{\vartheta}e^{\beta_{j}(s-\vartheta)} f_{j}(s)ds \nonumber \\
&=p_{i0}e^{-\beta_{i}\vartheta}+ e_{j0}e^{-\beta_{j}\vartheta}-p_{i0}e^{-\beta_{i}\vartheta}+ \int_{0}^{\vartheta}e^{\beta_{j}(s-\vartheta)} f_{j}(s)ds \nonumber \\
&=e_{j0}e^{-\beta_{j}\vartheta}+ \int_{0}^{\vartheta}e^{\beta_{j}(s-\vartheta)} f_{j}(s)ds =e_{j}(\vartheta) .
\end{align*}
The permissibility of the pursuer strategy (\ref{PSS}) is followed using lemma (\ref{TB}) as,
\begin{eqnarray}\label{kk25}
\|c_{i}(t)\|^{2}&=&\left\|\dfrac{(e^{\Lambda_{ij}\vartheta}e_{j0} -p_{i0})}{\vartheta }e^{-\beta_{i} t} + e^{\Lambda_{ij}(\vartheta - t)}f_{j}(t)\right\|^{2}  \nonumber \\
&=&\dfrac{e^{-2\beta_{i} t}\|e^{\Lambda_{ij}\vartheta}e_{j0} -p_{i0}\|^2}{\vartheta^2}+ 2 \left\langle \dfrac{(e^{\Lambda_{ij}\vartheta}e_{j0} -p_{i0})e^{-\beta_{i} t}}{\vartheta } , e^{\Lambda_{ij}(\vartheta - t)}f_{j}(t)\right\rangle  \nonumber \\
&+&\|e^{\Lambda_{ij}(\vartheta - t)}f_{j}(t)\|^{2} \nonumber \\
&\leq &\dfrac{e^{-2\beta_{i} t}\|e^{\Lambda_{ij}\vartheta}e_{j0} -p_{i0}\|^2}{\vartheta^2}+\frac{2e^{-\beta_{i} t}\|e^{\Lambda_{r}\vartheta}e_{j0} -p_{i0}\|e^{\Lambda_{ij}(\vartheta - t)}\|f_{j}(t)\|}{\vartheta}\nonumber\\
&+ & e^{2\Lambda_{ij}(\vartheta - t)}\|f_{j}(t)\|^{2} \nonumber \\
&\leq &\dfrac{e^{-2\beta_{i} t}\|e^{\Lambda_{ij}\vartheta}e_{j0} -p_{i0}\|^2}{\vartheta^2}+\frac{2e^{-\beta_{i} t}\|e^{\Lambda_{r}\vartheta}e_{j0} -p_{i0}\|e^{\Lambda(\vartheta - t)}\|f_{j}(t)\|}{\vartheta}\nonumber\\
&+ & e^{2\Lambda(\vartheta - t)}\|f_{j}(t)\|^{2} \nonumber \\
&\leq &  \dfrac{e^{-2\beta_{i} t}\|e^{\Lambda_{ij}\vartheta}e_{j0} -p_{i0}\|^2}{\vartheta^2}+\frac{2e^{-\beta_{i} t}\|e^{\Lambda_{ij}\vartheta}e_{j0} -p_{i0}\|\|f_{j}(t)\|}{\vartheta}  \nonumber\\
&+ & \|f_{j}(t)\|^{2} .
\end{eqnarray}
when $\beta\geq 0$, then from (\ref{kk25}), we have
\begin{align*}
\|c_{i}(t)\|^{2}&\leq \dfrac{\|e^{\Lambda_{ij}\vartheta}y_{j0} -x_{i0}\|^2}{\vartheta^2}+\frac{2\|e^{\Lambda_{ij}\vartheta}e_{j0} -p_{i0}\|\|f_{j}(t)\|}{\vartheta} + \|f_{j}(t)\|^{2}\\
&\leq \dfrac{\|e^{\Lambda_{ij}\vartheta}e_{j0} -p_{i0}\|^2}{\vartheta^2}+\frac{2\|e^{\Lambda_{ij}\vartheta}e_{j0} -p_{i0}\|\varsigma_{j} e^{k\vartheta}}{\vartheta}+\varsigma_{j}^{2} e^{2k\vartheta}\\
&\leq \dfrac{\|e^{\Lambda_{ij}\vartheta}e_{j0} -p_{i0}\|^2}{\vartheta^2}+\frac{2\|e^{\Lambda_{ij}\vartheta}e_{j0} -p_{i0}\|\varsigma e^{k\vartheta}}{\vartheta}+\varsigma^{2} e^{2k\vartheta}\\
&\leq \dfrac{\|e^{\Lambda_{ij}\vartheta}e_{j0} -p_{i0}\|^2}{\vartheta^2}+\frac{2\|e^{\Lambda_{ij}\vartheta}e_{j0} -p_{i0}\| e^{k\vartheta}}{\vartheta} \left(\varrho -\gamma\right)e^{-k\vartheta}
\\& +e^{-2k\vartheta}\left(\varrho -\gamma\right)^{2} e^{2k\vartheta}\\
&= \dfrac{\|e^{\Lambda_{ij}\vartheta}e_{j0} -p_{i0}\|^2}{\vartheta^2}+\frac{2\|e^{\Lambda_{ij}\vartheta}e_{j0} -p_{i0}\| }{\vartheta} \left(\varrho -\gamma\right) +\left(\varrho -\gamma\right)^{2} \\
&=\dfrac{\|e^{\Lambda_{ij}\vartheta}e_{j0} -p_{i0}\|^2}{\vartheta^2}+\frac{2\varrho\|e^{\Lambda_{ij}\vartheta}e_{j0} -p_{i0}\| }{\vartheta}-\frac{2\|e^{\Lambda_{r}\vartheta}e_{j0} -p_{i0}\| \gamma}{\vartheta}
\\&+\varrho^2 -2\varrho \gamma +\gamma^2 \\
&\leq\dfrac{\|e^{\Lambda_{ij}\vartheta}e_{j0} -p_{i0}\|^2}{\vartheta^2}+\frac{2\varrho\|e^{\Lambda_{ij}\vartheta}e_{j0} -p_{i0}\|}{\vartheta}-\frac{2\|e^{\Lambda_{ij}\vartheta}e_{j0} -p_{i0}\|^2}{\vartheta^2}
\\&+ \varrho^2 -\frac{2\varrho\|e^{\Lambda_{ij}\vartheta}e_{j0} -p_{i0}\|}{\vartheta}+\frac{\|e^{\Lambda_{ij}\vartheta}e_{j0} -p_{i0}\|^2}{\vartheta^2}\\
& = \varrho^2
 \leq \varrho_{i}^2 + 2k\int_{0}^{\vartheta}\|c_{i}(t)\|^{2}dt .
\end{align*}
when $\beta < 0$, then from (\ref{kk25}), we have
\begin{align*}
\|c_{i}(t)\|^{2}& = \dfrac{e^{-2\beta_{i} t}\|e^{\Lambda_{ij}\vartheta}e_{j0} -p_{i0}\|^2}{\vartheta^2}+\frac{2e^{-\beta_{i} t}\|e^{\Lambda_{ij}\vartheta}e_{j0} -p_{i0}\|\|f_{j}(t)\|}{\vartheta}\|f_{j}(t)\|^{2}\\
&\leq\dfrac{e^{2\phi t}\|e^{\Lambda_{ij}\vartheta}e_{j0} -p_{i0}\|^2}{\vartheta^2}+\frac{2e^{\phi t}\|e^{\Lambda_{ij}\vartheta}e_{j0} -p_{i0}\|\varsigma e^{\phi t}}{\vartheta}+\varsigma^{2} e^{2\phi t}\\
&\leq\dfrac{e^{2\phi \vartheta}\|e_{j0} -p_{i0}\|^2}{\vartheta^2}+\frac{2e^{2\phi \vartheta}\|e_{j0} -p_{i0}\|\varsigma}{\vartheta}+\varsigma^{2} e^{2\phi\vartheta}\\
&\leq\dfrac{e^{2\phi\vartheta}\|e^{\Lambda_{ij}\vartheta}e_{j0} -p_{i0}\|^2}{\vartheta^2}+\frac{2e^{2\phi\vartheta}\|e^{\Lambda_{ij}\vartheta}e_{j0} -p_{i0}\|}{\vartheta}\left(\varrho e^{-\phi\vartheta} -\gamma\right)+\left(\varrho e^{-\phi \vartheta} -\gamma\right)^{2} e^{2\phi\vartheta}\\
&=\dfrac{e^{2\phi\vartheta}\|e^{\Lambda_{ij}\vartheta}e_{j0} -p_{i0}\|^2}{\vartheta^2}+\frac{2e^{\phi\vartheta}\|e^{\Lambda_{ij}\vartheta}e_{j0} -p_{i0}\|\varrho}{\vartheta}-\frac{2e^{2\phi\vartheta}\|e^{\Lambda_{ij}\vartheta}e^{\Lambda_{ij}\vartheta}e_{j0} -p_{i0}\|\gamma}{\vartheta}
\\&+\varrho^2 -2\gamma\varrho e^{\phi \vartheta} +\gamma^2e^{2\phi \vartheta}\\
&\leq\dfrac{e^{2\phi\vartheta}\|e^{\Lambda_{ij}\vartheta}e_{j0} -p_{i0}\|^2}{\vartheta^2}+\frac{2e^{\phi\vartheta}\|e^{\Lambda_{ij}\vartheta}e_{j0} -p_{i0}\|\varrho}{\vartheta}-\frac{2e^{2\phi\vartheta}\|e^{\Lambda_{ij}\vartheta}e_{j0} -p_{i0}\|^2}{\vartheta^2}
\\&+\varrho^2 -\frac{2e^{\phi\vartheta}\|e^{\Lambda_{ij}\vartheta}e_{j0} -p_{i0}\|\varrho}{\vartheta}+\dfrac{e^{2\phi\vartheta}\|e_{j0} -p_{i0}\|^2}{\vartheta^2}\\
& = \varrho^2
 \leq \varrho_{i}^2 + 2k \int_{0}^{\vartheta}\|c_{i}(t)\|^{2}dt .
\end{align*}
Therefore, the proof of Theorem \ref{D3} is complete.
\end{proof}
To validate the proof of Theorem \ref{D3}, a numerical example is presented in the remark below.
\begin{remark} At the same time, the inequalities of the hypothesis of the Theorem \ref{D3} can be satisfied. Look at the following example:
\end{remark}
\begin{example}
 We take into consideration that $l = 2 $ and $m = 1$ for $\beta \geq 0 $. Then Theorem \ref{D1} inequalities can be fulfilled if
$ \beta_{1} = 3 ,\beta_{2}  = -4 , \varrho_{1}  = 250, \varrho_{2}  = 200, \varsigma_{1}  = \frac{1}{20} $,
$ \vartheta = 1 .$ and $ k = 3 $ 
Letting
$ e_{10} = ( e^{1}_{10} , e^{2}_{10}, ..., e^{n}_{10})  = (1,0,...,0 ) ,$
$ p_{10} = ( p^{1}_{10} , p^{2}_{10}, ..., p^{n}_{10})  = (2,0,...,0 ) $ \\
$ p_{20} = ( p^{1}_{20} , p^{2}_{20}, ..., p^{n}_{20}) = (3,0,...,0 ) .$ \\
This means that
$  \beta   = 3 $, $  \phi =  \displaystyle\max \{ 3 ,3 \} = 3 $ $\Lambda_{2,1} = -4-3 = -7. \ \ \ \Lambda = -7 $\\ and
$ \frac{\left\|(e_{10} -p_{10})\right\|}{\vartheta} = \frac{((-3)^2 +0^2 + ... +0^2 )^{\dfrac{1}{2}}}{1}   = 3.0$, thus
$\gamma = \displaystyle\max_{i,j} \{\frac{\|e^{\Lambda_{i,j}}e_{j0} -p_{i0}\|}{\vartheta}\} = 3.0.$
Consequently, we may observe that,
$$ \frac{1}{20} \leq e^{-3}\left(200  -3.0\right). $$
That is $$ \frac{1}{20} < 0.034,$$
which means that
 $$ \varsigma \leq e^{-k\vartheta}\left(\varrho -\gamma\right) .$$
The inequality for $\beta < 0$ can also be satisfied with the above assumed values, meaning that  $\beta = -4. $ and $  \phi =  \displaystyle\max \{ -4 ,3 \} = 4 $  \\
Therefore we can see that, $$\frac{1}{20} \leq 200 e^{-4} -3.0 .$$
That is $$ \frac{1}{20} < 0.66 .$$
This implies that,
$$ \varsigma \leq \varrho e^{-\phi\vartheta} -\gamma .$$
\end{example}
\begin{theorem}\label{D4}
 Suppose in the game (\ref{AT31})- (\ref{DAV22}) that
 $\Lambda >  0$ and $ \varrho > \gamma.$
If the $i^{th}$ pursuer uses the strategy (\ref{PSS})then pursuit can be completed whenever \\
  $$\varsigma \leq\left\{ \begin{array}{ll}
 e^{-\vartheta(k +\Lambda)}\left(\varrho -\gamma\right) , \beta \geq 0.   \\
 \left(e^{-\phi \vartheta}\varrho - \gamma \right)e^{-\Lambda\vartheta},\beta <0 . \end{array}\right.$$
\end{theorem}
\begin{proof}
The proof that $ p_{i}(\vartheta)= e_{j}(\vartheta)$ is analogous to that presented in Theorem \ref{D3}.\\
To establish the permissibility of strategy (\ref{PSS}), we proceed as follows:
\begin{eqnarray}\label{kk}
\|c_{i}(t)\|^{2}&=&\left\|\dfrac{(e^{\Lambda_{ij}\vartheta}e_{j0} -p_{i0})}{\vartheta }e^{-\beta_{i} t}
 + e^{\Lambda_{ij}(\vartheta - t)}f_{j}(t)\right\|^{2}  \nonumber \\
&=&\dfrac{e^{-2\beta_{i} t}\|e^{\Lambda_{ij}\vartheta}e_{j0} -p_{i0}\|^2}{\vartheta^2}+ 2 \left\langle \dfrac{e^{-\beta_{i} t}(e^{\Lambda_{ij}\vartheta}e_{j0} -p_{i0})}{\vartheta } , e^{\Lambda_{ij}(\vartheta - t)}f_{j}(t)\right\rangle  \nonumber \\
&+&\|e^{\Lambda_{ij}(\vartheta - t)}f_{j}(t)\|^{2} \nonumber \\
&\leq & \dfrac{e^{-2\beta_{i} t}\|e^{\Lambda_{ij}\vartheta}e_{j0} -p_{i0}\|^2}{\vartheta^2}+\frac{2e^{-\beta_{i} t}\|e^{\Lambda_{ij}\vartheta}e_{j0} -p_{i0}\|e^{\Lambda_{ij}(\vartheta - t)}\|f_{j}(t)\|}{\vartheta}\nonumber\\
&+ & e^{\Lambda_{ij}(\vartheta - t)}\|f_{j}(t)\|^{2} \nonumber \\
&\leq &  \dfrac{e^{-2\beta t}\|e^{\Lambda_{ij}\vartheta}e_{j0} -p_{i0}\|^2}{\vartheta^2}+\frac{2e^{-\beta t}\|e^{\Lambda_{ij}\vartheta}e_{j0} -p_{i0}\|e^{\Lambda\vartheta }\|f_{j}(t)\|}{\vartheta} \nonumber\\
&+ & e^{\Lambda\vartheta }\|f_{j}(t)\|^{2}.
\end{eqnarray}
when $\beta \geq 0$, then from (\ref{kk}), we have
\begin{align*}
\|c_{i}(t)\|^{2}&\leq \dfrac{\|e^{\Lambda_{ij}\vartheta}e_{j0} -p_{i0}\|^2}{\vartheta^2}+\frac{2\|e^{\Lambda_{ij}\vartheta}e_{j0} -p_{i0}\|e^{\Lambda\vartheta }\|f_{j}(t)\|}{\vartheta}+e^{2\Lambda\vartheta }\|f_{j}(t)\|^{2}\\
&\leq \dfrac{\|e^{\Lambda_{ij}\vartheta}e_{j0} -p_{i0}\|^2}{\vartheta^2}+\frac{2\|e^{\Lambda_{i,j}\vartheta}e_{j0} -p_{i0}\|\varsigma_{j} e^{\vartheta(\Lambda+k)}}{\vartheta}+\varsigma_{j}^{2} e^{2\vartheta(\Lambda +k)}
\end{align*}
\begin{align*}
&\leq \dfrac{\|e^{\Lambda_{i,j}\vartheta}e_{j0} -p_{i0}\|^2}{\vartheta^2}+\frac{2\|e^{\Lambda\vartheta}e_{j0} -p_{i0}\|\varsigma e^{\vartheta(\Lambda+k)}}{\vartheta}
+\varsigma^{2} e^{2\vartheta(\Lambda+k)}\\
&\leq \dfrac{\|e^{\Lambda_{ij}\vartheta}e_{j0} -p_{i0}\|^2}{\vartheta^2}+\frac{2\|e^{\Lambda_{ij}\vartheta}e_{j0} -p_{i0}\| e^{\vartheta(\Lambda+k)}}{\vartheta} \left(\varrho -\gamma\right)e^{-\vartheta(\Lambda+k)}
\\&+e^{-2\vartheta(\Lambda +k)}\left(\varrho -\gamma\right)^{2} e^{2\vartheta(\Lambda+k)}\\
&= \dfrac{\|e^{\Lambda_{ij}\vartheta}e_{j0} -p_{i0}\|^2}{\vartheta^2}+\frac{2\|e^{\Lambda_{ij}\vartheta}e_{j0} -p_{i0}\| }{\vartheta} \left(\varrho -\gamma\right) +\left(\varrho -\gamma\right)^{2} \\
&=\dfrac{\|e^{\Lambda_{ij}\vartheta}e_{j0} -p_{i0}\|^2}{\vartheta^2}+\frac{2\varrho\|e^{\Lambda_{ij}\vartheta}e_{j0} -p_{i0}\| }{\vartheta}-\frac{2\|e^{\Lambda_{ij}\vartheta}e_{j0} -p_{i0}\| \gamma}{\vartheta}
\\&+\varrho^2 -2\varrho \gamma +\gamma^2 \\
&\leq\dfrac{\|e^{\Lambda_{ij}\vartheta}e_{j0} -p_{i0}\|^2}{\vartheta^2}+\frac{2\varrho\|e^{\Lambda_{ij}\vartheta}e_{j0} -p_{i0}\|}{\vartheta}-\frac{2\|e^{\Lambda_{ij}\vartheta}e_{j0} -p_{i0}\|^2}{\vartheta^2}
\\&+ \varrho^2 -\frac{2\varrho\|e^{\Lambda_{i,j}\vartheta}e_{j0} -p_{i0}\|}{\vartheta}+\frac{\|e^{\Lambda_{ij}\vartheta}e_{j0} -p_{i0}\|^2}{\vartheta^2}\\
&= \varrho^2
\leq \varrho_{i}^2 + 2k\int_{0}^{\vartheta}\|c_{i}(t)\|^{2}dt .
\end{align*}
when $\beta < 0$, then from (\ref{kk}), we have
\begin{align*}
\|c_{i}(t)\|^{2} & = \dfrac{e^{2\phi \vartheta}\|e^{\Lambda_{ij}\vartheta}e_{j0} -p_{i0}\|^2}{\vartheta^2}+\frac{2e^{2\phi \vartheta}\|e^{\Lambda_{ij}\vartheta}e_{j0} -p_{i0}\|e^{\Lambda\vartheta}\left(e^{-\phi \vartheta}\varrho - \gamma \right)e^{-\Lambda\vartheta}}{\vartheta}
\\&+e^{2\Lambda\vartheta }\left(e^{-\phi \vartheta}\varrho - \gamma \right)^{2}e^{-2\Lambda\vartheta} e^{2\phi \vartheta}\\
&\leq \gamma^{2} e^{2\phi\vartheta}+2\gamma e^{2\phi \vartheta} \left( e^{-\phi \vartheta}\varrho - \gamma \right)+\left( e^{-\phi \vartheta}\varrho - \gamma\right)^{2}e^{2\phi \vartheta}\\
&\leq e^{2\phi\vartheta}\left(\gamma^{2} +2\gamma  \left( e^{-\phi \vartheta}\varrho - \gamma \right)+\left( e^{-\phi \vartheta}\varrho - \gamma\right)^{2} \right)\\
&= e^{2\phi\vartheta}\left(\gamma^{2} +2\gamma   e^{-\phi \vartheta}\varrho  -2 \gamma^{2} + e^{-2\phi \vartheta}\varrho^{2} -2\gamma\varrho e^{-\phi \vartheta} +  \gamma^{2} \right)\\
&= \varrho^2
 \leq \varrho_{i}^2 + 2k \int_{0}^{\vartheta}\|c_{i}(t)\|^{2}dt .
\end{align*}
Therefore, the proof of Theorem \ref{D4} is complete.
\end{proof}
To validate the proof of Theorem \ref{D4}, a numerical example is presented in the remark below.
\begin{remark} The inequality in Theorem \ref{D4} can be satisfied concurrently. Here is an example:
\end{remark}
\begin{example}
 For $  \beta \geq 0 $ ,we consider $l =2 $ and $m = 1$. Then the inequalities in Theorem \ref{D1} can be satisfied by choosing
$ \beta_{1} = 3 ,\beta_{2}  = -4 , \varrho_{1}  = 2050000, \varrho_{2}  = 100000, \varsigma_{1}  = \frac{1}{20} $ ,
$ \vartheta = 1 ,$ $ k = 3 $ \\
and allowing
$ e_{10} = ( e^{1}_{10} , e^{2}_{10}, ..., e^{n}_{10})  = (1,0,...,0 ) ,$
$ p_{10} = ( p^{1}_{10} , p^{2}_{10}, ..., p^{n}_{10})  = (2,0,...,0 ) $ \\
$ p_{20} = ( p^{1}_{20} , p^{2}_{20}, ..., p^{n}_{20}) = (3,0,...,0 ) .$ \\
This means that
$  \beta   = 3 $,  $\Lambda_{1,2} = 3 +4 = 7. \ \ \ \Lambda = 7 $\\ and
$ \frac{\left\|(e_{10} -p_{10})\right\|}{\vartheta} = \frac{((-3)^2 +0^2 + ... +0^2 )^{\dfrac{1}{2}}}{1}  1094$, thus
$\gamma = \displaystyle\max_{i,j} \{\frac{\|e^{\Lambda_{i,j}}e_{j0} -p_{i0}\|}{\vartheta}\} = 1094.$
Therefore we can see that,
$$\frac{1}{20} \leq e^{-1(10)}\left(100000 -1094\right). $$
That is $$ \frac{1}{20} < 1811.53,$$
which means that
$$ \varsigma \leq  e^{-\vartheta(k +\Lambda)}\left(\varrho -\gamma\right)$$
 With assumed values above, the inequality for $ \beta < 0$ can also be satisfied, this means that
$  \beta   = -4. $ and $  \phi =  \displaystyle\max \{ -4 ,3 \} = 3 $  \\
Therefore, we can see that, $$\frac{1}{20} \leq \left(e^{-3 }100000 - 1094 \right)e^{-7} .$$
That is $ \frac{1}{20} < 0.67, $
which means that
 $ \varsigma \leq \left(e^{-\phi \vartheta}\varrho - \gamma \right)e^{-\Lambda\vartheta}.$
\end{example}
 \section{Conclusion}
\noindent We investigated a class of multi-agent pursuit problems over a fixed operational time horizon in the space $\mathbb{R}^n$. The motion of each agent is governed by a first-order linear differential equation.\\
\noindent We established Theorems \ref{D1}, \ref{D3}, and \ref{D4} under different operational conditions characterized by the parameter $\Lambda$ (defined as the maximum of $\beta_i - \beta_j$). These results provide explicit criteria under which successful interception is guaranteed.\\
\noindent In Theorem \ref{D1}, we showed that when $\Lambda = 0$ and the maximum speed of the pursuers exceeds $\gamma$, interception is ensured. In Theorem \ref{D3}, we proved that if $\Lambda < 0$ and $\gamma$ is less than the maximum speed of the pursuers, pursuit is also completed. Finally, Theorem \ref{D4} demonstrates that when $\Lambda > 0$ and the pursuers’ maximum speed exceeds $\gamma$, successful interception still occurs.\\
\noindent To illustrate the practical relevance of the theoretical findings, we presented numerical examples that confirm the analytical conditions for successful pursuit.\\
This work extends and generalizes earlier results by Ahmed \cite{REFM1} and Gbande et al. \cite{REFM4}. The theoretical framework developed here has potential applications in engineering systems, economic competition models, missile guidance, autonomous vehicles, robotics, and related multi-agent control problems.

\section{Declarations}
\textbf{Competing Interest}: The authors declare no conflicts of interest.\\
\textbf{Funding}: This research received no funding.\\
\textbf{Availability of data and material}: No data is associated with this work.\\

\end{document}